\documentclass{article}
\usepackage{amsfonts,amsmath,amsthm,amssymb}
\usepackage{graphics, epsfig}
\usepackage{color}
\usepackage{appendix}
\usepackage{ulem}
\usepackage[makeroom]{cancel}
\usepackage{fancyhdr}
\usepackage{centernot}
\usepackage{mathtools}
\usepackage{ stmaryrd }

 \usepackage[usenames,dvipsnames]{pstricks}
 \usepackage{pst-grad} 
 \usepackage{pst-plot} 
\allowdisplaybreaks

\let\TeXchi\chi
\newbox\chibox
\setbox0 \hbox{\mathsurround0pt $\TeXchi$}
\setbox\chibox \hbox{\raise\dp0 \box 0 }
\def\chi{\copy\chibox}

\newtheorem{proposition}{Proposition}[section]
\newtheorem{theorem}{Theorem}[section]
\newtheorem{definition}{Definition}[section]
\newtheorem{example}{Example}[section]
\newtheorem{lemma}{Lemma}[section]

\newtheorem{remark}{Remark}[section]

\numberwithin{equation}{section}
\numberwithin{theorem}{section}
\numberwithin{definition}{section}
\numberwithin{example}{section}
\numberwithin{proposition}{section}
\numberwithin{lemma}{section}
\numberwithin{remark}{section}
\begin{document}
\title{On structured spaces and their properties}
\author
{Manuel Norman}
\date{}
\maketitle
\begin{abstract}
\noindent In this paper we introduce a new kind of topological space, called 'structured space', which locally resembles various kinds of algebraic structures. This can be useful, for instance, to locally study a space that cannot be globally endowed with an algebraic structure using tools from algebra. The definition of these spaces will be made more precise via one of our main result, which involves the 'structure map'. This will also lead us to a rigorous and unambiguous definition of algebraic structure. After showing some examples which naturally arise in this context, we study various properties and develop some theory for these new spaces; in particular, we consider partitions (with respect to some measure $\mu$). We then prove one of the most important Theorem of this paper (Theorem \ref{Thm:4.1}), which states that every structured space, under some assumptions, induces a lattice, and conversely every lattice induces a structured space satisfing such hypothesis. We conclude with some relations with connected spaces.
\end{abstract}
\let\thefootnote\relax\footnote{Author: \textbf{Manuel Norman}; email: manuel.norman02@gmail.com\\
\textbf{AMS Subject Classification (2010)}: 08A05, 22A30\\
\textbf{Key Words}: structured space, algebraic structure, structure map, lattice}
\section{Introduction}
\noindent We first recall some basic facts about algebraic structures (see also the references). When we say 'algebraic structure', we usually refer to a set endowed with some operations (and satisfing some properties). Here, we will usually assume that an algebraic structure has at least one operation and satisfies at least one property (to better understand this last part, we will need the structure map). Among the possible examples, we have groups and group-like structures (magmas, semigroups, loops, monoids, ...),  rings and ring-like structure (semirings, integral domains, fields, ...), algebras, ... We can also add some other structures, obtaining, for instance, topological groups, Lie algebras, $C^*$-algebras (where we also define a norm), ...\\
For now, we will use this idea of algebraic structure to develop our basic theory. We will then give a precise definition of algebraic structure, and also of structured space.\\
We start with the following definition of neighborhood (sometimes, other definitions are considered; here a neighborhood in not assumed to be open, as Definition \ref{Def:1.1} shows):
\begin{definition}\label{Def:1.1}
Given a topological space $X$ (say, with topology $\tau$), we say that $U_p \subseteq X$ is a neighborhood of $p \in X$ if it contains an open set which contains $p$.
\end{definition}
We can now give a first definition of what a structured space is (later in this section, we will give a reason why these spaces are interesting to study):
\begin{definition}[Structured space; first version] \label{Def:1.2}
Let $X$ be a topological space. If every point $p\in X$ has at least one neighborhood $U_p \varsupsetneq \lbrace p \rbrace$ which is an algebraic structure, then $X$ is said to be a structured space.
\end{definition}
This first definition is still not so precise; the aim of the first part of this paper is to make unambiguous this notion. We first note that, when we consider a structured space, we also fix some neighborhood $U_p$ (and also their chosen algebraic structure) for every $p \in X$. This means that, if for instance $p$ has two neighborhoods $U_p$, $\widehat{U}_p$ which can be assigned the structure of a ring or a group, and of a vector space, we choose and fix one of these neighborhoods and one the possible structures. So, a structured space has also some assigned (fixed) neighborhoods $U_p$ with fixed algebraic structures. We can thus consider (just for now; this will be improved with the structure map) a structured space as a couple $(X, \lbrace U_p \rbrace )$ with fixed $U_p$'s and fixed structures. This distinguishes among structured spaces obtained from the same topological space $X$ but choosing different neighborhoods and structures, and it is indeed more precise than Definition \ref{Def:1.2}. We also notice that we assume that the fixed $U_p$'s have more than one point: this way we avoid "trivial" algebraic structures.\\
We note that the definition of structured space says that $X$ can be locally represented by different algebraic structures. It is also clear that every algebraic structure endowed with any topology is a structured space: just take $U_p \equiv X$ for every $p \in X$. Another important note is that, when the local structures of the fixed $U_p$'s require a topology, we can use different topologies from the one on $X$, temporary dropping the global topological structure of $X$ and considering only the local topological structure of $U_p$. Another way to see this is by using the notion of bitopological space introduced by Kelly in [6]. We will usually consider, however, the temporary drop of the structure when necessary.\\
We give a simple example of construction of structured space:
\begin{proposition}\label{Prop:1.1}
Consider a collection of algebraic structure $\mathcal{C}=\lbrace A_j \rbrace$. Letting $X:=\bigcup_j A_j$, if we define on $X$ a topology $\tau$ such that every $A_j$ belongs to $\tau$, then $X$ is a structured space (a method to construct a topology with these properties is shown in Example \ref{Ex:1.1}).
\end{proposition}
\begin{proof}
It is clear that, if $x \in X$, we can always choose at least one (open, by the assumptions of this Proposition) neighborhood of $x$ with a certain algebraic structure (and then we fix it). Hence, the topological space $X$ is a structured space.
\end{proof}
\begin{example}\label{Ex:1.1}
\normalfont Consider the sets $[0,1]$, $GL(n, \mathbb{R})$ ($n \geq 1$), $L^2(\mathbb{R}, \lambda)$ (where $\lambda$ denotes Lebesgue measure) endowed with the following algebraic structures, respectively: magma, Lie group, vector space. Then, consider the space $X:=[0,1] \cup GL(n, \mathbb{R}) \cup L^2(\mathbb{R},\lambda)$. We define on it a topology as follows: we start to construct $\tau$ by letting it have $X$, $\emptyset$, $[0,1]$, $GL(n, \mathbb{R})$, $L^2(\mathbb{R}, \lambda)$. We then need that every finite intersection of element of $\tau$ still belongs to $\tau$, and also that every (also infinite) union of elements of $\tau$ still belongs to $\tau$. So first add to $\tau$:\\
$$[0,1] \cup GL(n, \mathbb{R})$$\\
$$[0,1] \cup L^2(\mathbb{R}, \lambda)$$\\
$$L^2(\mathbb{R}, \lambda) \cup GL(n, \mathbb{R})$$\\
(In a more general setting, consider all the possible unions of the elements added before this step). It is clear that the union of the sets added up to now to $\tau$ always belongs to $\tau$. Call the collection with only these sets $\mathcal{A}$. Now consider the collection:
$$\mathcal{B}:=\lbrace \, \bigcap_{j=1}^{k} A_j, A_j \in \, \mathcal{A}, k \in \mathbb{N} \setminus \lbrace 0 \rbrace, k<+\infty \rbrace$$
If we define now the collection $\tau$ of all the possible (also infinite) unions of elements in $\mathcal{B}$, it is clear that this is a topology (it is the smallest topology on $X$ such that the collection considered at the beginning (the one consisting of the fixed neighborhoods together with $X$ and $\emptyset$) is contained in it). Indeed, $X$ and $\emptyset$ belong to $\tau$ by construction, and always by construction this collection is closed under (also infinite) unions. Then, since intersections of unions are unions of intersections, we can conclude that finite intersections of sets in $\tau$ belong to $\tau$, which ends the proof that $\tau$ is a topology. Hence (by what we said at the beginning of this example, and by the previous Proposition) $X$ is a structured space.
\end{example}
\begin{example}\label{Ex:1.2}
\normalfont We also give an example of a structured space with "the same kinds of elements", in this case matrices. Let:
$$X:= \bigcup_{n \in \mathbb{N} \setminus \lbrace 0 \rbrace} GL(n, \mathbb{R})$$
and assign some algebraic structures to each of these sets. For instance, consider all the $GL(n, \mathbb{R})$ to be groups when $n$ is odd, and to be Lie groups when $n$ is even. It is clear that, by considering the topology $\tau$ of Example \ref{Ex:1.1} (which can be easily generalised also to these cases) on $X$, we conclude by Proposition \ref{Prop:1.1} that this space is a structured space.
\end{example}
We now give a reason to introduce this new kind of space: it is clear, in particular by Example \ref{Ex:1.2} (and also the ones below), that many spaces cannot be endowed with a certain algebraic structure \textit{globally} (in the said example, we cannot sum matrices with different orders), but they could be endowed \textit{locally} with various kinds of structures. Furthermore, spaces like the one in Example \ref{Ex:1.2} can be interesting to study for various purposes. Thus, a space which cannot be endowed globally with a certain structure but which is important, for some reasons, to study, could be endowed locally with various structures. Then, we can use tools from algebra to analyse \emph{locally} these spaces.
\begin{example}\label{Ex:1.3}
\normalfont Based on this motivation and on Example \ref{Ex:1.2}, we give other important examples of structured spaces. We recall the following notations:\\
$\bullet$ $SL(n, F)$ is the special linear group of (matrices of) degree $n$ over the field $F$;\\
$\bullet$ $O(p,q)$ is the indefinite orthogonal (Lie) group;\\
$\bullet$ $SO(p,q)$ is the indefinite special orthogonal group;\\
$\bullet$ $O(n,F)$ is the (general) orthogonal group;\\
$\bullet$ $SO(n,F)$ is the special orthogonal group;\\
$\bullet$ $Sp(n,F)$ is the sympletic group;\\
$\bullet$ $SU(n)$ is the special unitary (Lie) group; ...\\
It is clear that all these well known groups are important, and it is also clear that $\bigcup_{n \in \mathbb{N} \setminus \lbrace 0 \rbrace} SL(n,F)$, $\bigcup_{n \in \mathbb{N} \setminus \lbrace 0 \rbrace} SO(n,F)$, ... are important spaces. However, these unions cannot be clearly groups under the same operations of the "components" $SL(n,F)$, $SO(n,F)$, ... (we cannot sum matrices with different orders). Since these spaces are important, it would be interesting to study them \textit{locally} with some kinds of algebraic structures. The notion of structured spaces is then really useful in such situations. Defining a topology as in Proposition \ref{Prop:1.1}, all the above unions become structured spaces. Moreover, also the space $\bigcup_{n \in \mathbb{N} \setminus \lbrace 0 \rbrace} \mathbb{R}^n$ is important to study, but it is clear that it is not a "classical" algebraic structure (we cannot sum vectors with different numbers of components). Even in this case, it can be easily seen that this is a structured space (and since $\mathbb{R}^n$ can be endowed with many different kinds of structures, we can consider various local structures and obtain many different structured spaces, all with the same topological space $\bigcup_{n \in \mathbb{N} \setminus \lbrace 0 \rbrace} \mathbb{R}^n$).
\end{example}
\begin{example}\label{Ex:1.4}
\normalfont A really interesting example of structured space is the following one. Consider the complex projective spaces $\mathbb{CP}^n$ (see [37-48] for more on projective spaces). It is known (see [49-50]) that a cell structure (for CW-complexes, see [52-59]) for $\mathbb{CP}^n$ can be constructed using the fact that $\mathbb{CP}^n \setminus \mathbb{CP}^{n-1}$ is homeomorphic (for $n=1, ...$) to an open $2n$-ball in $\mathbb{R}^{2n} \cong \mathbb{C}^n$ ($\cong$ indicates isomorphism here). Since a $2n$-open ball is homeomorphic to $\mathbb{R}^{2n} \cong \mathbb{C}^n$, we can say that:
$$ \mathbb{CP}^n = \mathbb{C}^n \cup \mathbb{C}^{n-1} \cup ... \cup \mathbb{CP}^1 $$
$\mathbb{CP}^1$ is homeomorphic to the Riemann sphere, and it can be written (via the homeomorphism above) as $\mathbb{C} \cup \lbrace \infty \rbrace$. We would like to \textit{locally} endow these spaces with some algebraic structures. Consider all the $\mathbb{C}^n$, $n \geq 2$ as vector spaces over $\mathbb{C}$. The unique problem is the Riemann sphere $\mathbb{CP}^1$. We have two possibilities:\\
1) We define 1-generalised structured spaces as topological spaces $X$ such that $X \setminus \lbrace \phi \rbrace$ (for some fixed $\phi \in X$) is a structured space. 1-generalised spaces are not structured spaces because they have a fixed neighborhood, namely $\lbrace \phi \rbrace$ (to which we assign some kind of algebraic structure), with only one point. It is not difficult to see, however, that many results for structured spaces also hold for 1-generalised spaces. Theorem \ref{Thm:4.1} is not anymore available, because $h$ cannot be surjective if a fixed neighborhood contains only one point (see Section 4). We will not usually deal with 1-generalised spaces, but they are clearly similar to structured spaces, and so they are also interesting to study. Returning to our example, $\mathbb{CP}^1$ can be written as $\mathbb{C} \, \cup \, \lbrace \infty \rbrace$, and hence we can endow the $\mathbb{C}^n$'s with the structure of vector spaces over $\mathbb{C}$, while $\lbrace \infty \rbrace$, which is a 1-generalised structured space, can be endowed with any algebraic structure, say, for instance, a magma with operation $+$, where the unique property considered is closure under $+$ (define $\infty + \infty := \infty$). Consequently, it is clear that $\mathbb{CP}^n$ is, for all $n \geq 0$, a 1-generalised structured space.\\
2) The other possiblity is to add an element $\xi$ to $\mathbb{CP}^1$, where $0/0:= \xi$, so that $\mathbb{CP}^1 \cup \lbrace \xi \rbrace$ is a wheel (see [51] for this kind of algebraic structure). This way, considering $\mathbb{C}^n$ for $n \geq 2$ as vector spaces over $\mathbb{C}$, $\mathbb{CP}^n \, \cup \, \lbrace \xi \rbrace$ is a structured space in the usual sense.\\
Thus, these spaces can be viewed as structured spaces/1-generalised structured spaces. Moreover, the space $\mathbb{CP}^{\infty}$, given as the union $\bigcup_{n \geq 1}  \mathbb{CP}^n$ (or as the direct limit of these spaces), plays an important role in various situations. For instance, $\mathbb{CP}^{\infty}$ can be seen as the classifying space $BU(1)$ of $U(1)$, and also as the Eilenberg-MacLane space $K(\mathbb{Z},2)$ (see [60-76]). It is clear, by its definition and by the arguments above, that this space is a 1-generalised structured space or, adding a point $\xi$, as a structured space in the usual sense. Consequenly, $BU(1)$ and $K(\mathbb{Z},2)$ can be also studied as either 1-generalised structured spaces or, adding a point $\xi$, as structured spaces.
\end{example}
\begin{example}\label{Ex:1.5}
\normalfont Consider the tangent bundle of a differentiable manifold $M$ (see [3]):
$$ TM:= \bigsqcup_{x \in M} T_x M $$
where $T_x$ is the tangent space to $M$ at $x$. It is well known that this is an example of vector bundle. Notice that, since $T_x M$ can be endowed with the structure of a vector space for all $x \in M$, each set $T_x M \times \lbrace x \rbrace$ (which appears by definition of disjoint union) can be endowed with the same structure: the definition of sum and scalar multiplication simply does not affect the second component (which is always $x$), that is, if $a,b \in T_x M$, and $\alpha$ is a scalar in the considered field, the sum in $T_x M \times \lbrace x \rbrace$ is defined by $(a,x) + (b,x) := (a+b,x)$ and $\alpha (a,x) := (\alpha a, x)$. This is clearly a vector space over the same field as before. Consequently, the tangent bundle is a structured space. Actually, something similar holds in a more general setting: if we consider any vector bundle $(E, X, \pi)$, by definition all its fibers $\pi^{-1}(\lbrace x \rbrace)$ must be vector spaces. Furthermore, since $\pi : E \rightarrow X$ is a (continuous) surjection, we can write $E$ as the union of all its fibers. It is then clear that the total space $E$ is a structured space (notice that $TM$ is the total space of the tangent bundle).
\end{example}
In the next Section we start to work towards a rigorous definition of structured space (and hence also of algebraic structure, by the observation above). In particular, we will define a map which assigns to each fixed neighborhood $U_p$ its local structure. In order to do this, we need to encode three kinds of description of an algebraic structure: the operations defined on $U_p$, the properties satisfied by the structure (e.g. commutativity, ...), and the additional non-algebraic structures (for instance, for Lie groups we also need to describe somehow the fact that it is also a smooth manifold with smooth operations, ...).
\section{The structure map}
We will divide into three subsections the descriptions we referred to at the end of Section 1. In the fourth subsection we will unify them and finally give, via the structure map, a precise definition of algebraic structure and structured space. This will be one of the main result of this paper.\\
Before starting, we notice that the idea of the first four subsection is to assign to each $U_p$ the algebraic structure we want to define on it. However, there is also another possibility: we could assign to each $p$ the algebraic structure defined on the corresponding (chosen) neighborhood $U_p$. This is useful in particular when $U_p = U_q$ for some $p \neq q$ (which can happen quite often). We will analyse this in Section 2.5, as a consequence of the preceding subsections.
\subsection{Operations}
We start by describing one of the fundamental aspects of an algebraic structure: the operation/s defined on it. The idea is to define a map $g_1$ from every fixed neighborhood to the corresponding operations defined on it:
$$U_p \xmapsto{g_1} \, \text{operations defined on} \, U_p$$
We can thus think to write a tuple containing all the operations defined on $U_p$. For instance, if $R$ is a ring with the operations given by '$+$' and '$\cdot$', then we can assign to $R$ the couple $(+,\cdot)$. Furthermore, we need $(+, \cdot)$ to be the same as $(\cdot, +)$, since we do not want to define an 'ordered pair' of operations. To do this, we actually just need to (arbitrarly) choose one of the equivalent tuples of operations, and use only that one (it does not matter which tuple we choose among all these equivalent tuples). Thus, consider the collection:
$$ \mathcal{T}_1:=\lbrace \, \text{tuple of operations defined on} \, U_p, \, \text{with} \, U_p \, \text{fixed local structure of} \, X \rbrace / \sim_1 $$
where $\sim_1$ formalises the previous cutting process: $(+,\cdot) \sim_1 (\cdot, +)$ (which is easily generalised to any tuple of operations). It is clear that $\sim_1$ is an equivalence relation, since it is reflexive (for instance, $(+,\cdot) \sim_1 (+,\cdot)$), it is symmetric (for instance, $(+,\cdot) \sim_1 (\cdot, +)$ $\Rightarrow$ $(\cdot, +) \sim_1 (+, \cdot)$) and it is transitive (for instance, $(+,\cdot) \sim_1 (\cdot, +)$ and $(\cdot, +) \sim_1 (+, \cdot)$ imply $(+,\cdot) \sim_1 (+,\cdot)$; this is better seen with more than two operations); hence, the quotient above is w.r.t. (with respect to) the equivalence relation $\sim_1$.\\
The collection $\mathcal{T}_1$ describes all the operations on every local structure in $X$. We define
$$g_1: \mathcal{U} \rightarrow \mathcal{T}_1$$
by
\begin{equation}\label{Eq:2.1}
U_j \xmapsto{g_1} ( \cdot_1, ...)
\end{equation}
where
$$ \mathcal{U}:= \lbrace \, \text{fixed neighborhoods} \, U_p \text{'s} \, \text{of} \, X \rbrace $$
It is clear, by the definition of $\mathcal{T}_1$, that $g_1$ is surjective.\\
This map $g_1$ precisely describes the operations defined on every local structure $U_p$. Hence, if we want to precisely specify what kind of operations on $U_p$ we want to deal with (i.e. we want to fix the operations we will consider on $U_p$), we can use this map $g_1$. This is the first key passage in the definition of the structure map (Section 2.4).\\
\begin{example}\label{Ex:2.1}
\normalfont Consider a vector space $V$ over a field $F$. If we want to describe the operations on this structured space (which is actually an algebraic structure), we assign to $V$ the couple of operations $(+, \cdot_F)$, where $+$ is the sum, and $\cdot_F$ is the multiplication for a scalar in the field $F$:
$$ V \xmapsto{g_1} (+, \cdot_F) $$
\end{example}
\subsection{Properties}
We now proceed in our rigorous description of the local structures considering the properties they satisfy. The idea is always to assign something to the local structure that represents the properties considered. This will be accomplished by using the 'encoding function' of some property.\\
We note that here we shall be more careful: the reason is explained with the following example. Every group is a magma, in the same way as any group is also a semigroup. So, if $U_p$ is a magma, but it also turns out that it is a group, what should we assign to $U_p$? The answer is actually simple: if we want to consider $U_p$ to be just a magma, we will assign it only the property of totality (which is equivalent to closure), while if we want to consider it as a group we will assign it the properties satisfied by a group. We can justify this noting that, if for instance we have a result which holds for magmas, and if $G$ is a group, then we consider $G$ to be just a magma and we appy such result to $G$. Thus, the map that will assign to $U_p$ the properties \textit{we want it to have} will also tell us what kind of structure we want to consider on it (in our previous example, if we want to consider $G$ only as a magma we assign it only the property of closure; if we want to consider it as a group, we assign it the properties of a group). Note that, this way, to $U_p$ we assign only one element encoding the properties we want to consider on it, and not all the elements which encode 'sub-properties', ... Therefore, this will give a rigorous and unambiguous description of the fixed local structure.\\
After this important observation, we start to define the encoding functions of the properties of the considered local structures. We do this by giving some examples for some common properties.\\
Consider \emph{commutativity} with respect to a certain operation, say $\cdot_k$. This property assures that:
$$ x \cdot_k y = y \cdot_k x$$
Define the right $\cdot_k$ operator by:
$$r_{\cdot_k}[x](y):= y \cdot_k x$$
The left $\cdot_k$ operator $l_{\cdot_k}[x]$ is defined analogously:
$$l_{\cdot_k}[x](y):= x \cdot_k y$$
Then, commutativity says that:
$$f(x,y):=l_{\cdot_k}[x](y)-r_{\cdot_k}[x](y) \equiv 0$$
over the structure $A$ on which the operation is defined (i.e. for every $x,y \in A$). This function $f$ encodes the property of commutativity over $A$. It is clear that, if we consider $f$ over another set $B$ with the same operations but that does not satisfy commutativity w.r.t. $\cdot_k$, then $f(x,y)\not \equiv 0$. This function $f$ is thus called the 'encoding function' of the property of commutativity w.r.t. $\cdot_k$. It is important to notice that there is no reason to assume that the operation '$-$' is a priori defined in every case. In fact, we define it formally in such a way that $a-b=0$ $\Leftrightarrow$ $a=b$. For the other cases (i.e. $a \neq b$), the operation '$-$' can be defined formally in any way (it is not important how we define it, since we are only interested in the case $a=b$). Henceforth, we always consider, when dealing with the encoding function, that '$-$' is defined formally in such a way that $a-b=0$ $\Leftrightarrow$ $a=b$. More precisely, we will always assume the following formal definition: if $a,b \in A$, we define '$-$' by just saying that the result of $a-b$ is (formally) $a-b$ (without actually assigning it a specific value) whenever $a \neq b$, while it is $0$ whenever $a=b$. This is sufficient for our aims here.\\
Before giving other examples and the precise definition, we note that we will consider the encoding function as a map that encodes \emph{only one property}. This will avoid ambiguity. So, if for instance we want to encode commutativity w.r.t. $\cdot_k$ and also w.r.t. $\cdot_j$, we will do this \textit{separately} by considering the two encoding functions of each property.\\
Another often used property is the existence of the \emph{identity element}. Following what we did for commutativity, consider:
$$ _r f_{id}[y](x):=x \cdot_k y - x $$
and
$$ _l f_{id}[y](x):=y \cdot_k x - x $$
Consider then the vector valued map
$$ f_{id}[y](x):=(_l f_{id}[y](x), _r f_{id}[y](x)) $$
Then, if $\exists \widehat{y} \in A$ such that $_r f_{id}[\widehat{y}](x) \equiv 0$ $\forall x \in A$, the structure $A$ has a right identity element. A similar reasoning holds for the left identity element. If a structure has both a right and left identity element, then $f_{id}[\widehat{y}](x) \equiv 0$ over $A$. Hence, these three maps are the encoding functions of the right/left identity element, and of the identity element, respectively.\\
Another commonly used property is \emph{associativity} w.r.t. one operation, say $\cdot_k$. To encode this, first define:
$$ r(x,y,z):= (x \cdot_k y) \cdot_k z$$
and
$$ l(x,y,z):= x \cdot_k (y \cdot_k z)$$
Now consider:
$$f(x,y,z):=r(x,y,z)-l(x,y,z)$$
Then, if $f(x,y,z) \equiv 0$ $\forall x,y,z \in A$, the structure $A$ is associative w.r.t. $\cdot_k$. Thus, this $f$ is the encoding function of associativity w.r.t. $\cdot_k$.\\
Our last example is related to closure (which turns out to be equivalent to totality in some cases) w.r.t. some operation, say $\cdot_k$. Suppose, for instance, that $\cdot$ is a binary operation on $A$. Then $\cdot: S \subseteq A \times A \rightarrow A$, because actually this can be in general a partial function, i.e. its domain is contained in $A \times A$, and does not necessarily coincide with it. It is clear that totality (i.e. $S \equiv A \times A$) is equivalent to closure here, so we just need to encode this fact. This can be done by considering the map $f_c$:
$$ \cdot \xmapsto{f_c} A \times A \setminus S $$
It is clear that, if $f_c(\cdot) \equiv \emptyset$, which is "the zero in the operations of union and intersection of sets", then we have closure. Thus, this function encodes closure (and totality).\\
We finally give a definition of 'encoding function':
\begin{definition}\label{Def:2.1}
Consider an algebraic structure $A$. An encoding function of a certain property is a function $f$ whose domain depends on $A$ (meaning that it could be $A^2$, ... or in general is related somehow to $A$) such that this function is precisely the zero function (w.r.t. $A$) if and only if the structure $A$ satisfies the property encoded by $f$.
\end{definition}
\begin{remark}\label{Rm:2.1}
\normalfont It is possible that several encoding functions exist for the same property. For instance, consider the example above related to commutativity. Instead of the encoding function
$$f(x,y):=l_{\cdot_k}[x](y)-r_{\cdot_k}[x]$$
we could consider the encoding function (we suppose that in the following example it is defined a multiplication by the scalar $-1$):
$$g(x,y):=-f(x,y)=r_{\cdot_k}[x](y)-l_{\cdot_k}[x](y)$$
This is not however a problem. As we did in Section 2.1, we define an equivalence relation $\sim$ which says that: $f \sim g$ iff they encode the same property. This is easily verified to be, indeed, an equivalence relation.\\
Also note that the encoding function itself does not depend on a particular structure: we assign it a domain which is some kind of algebraic structure, and if the encoding function turns out to be $\equiv 0$ over such structure then the structure satisfies that property, otherwise it does not satisfy it.
\end{remark}
We can now define a map $g_2$ which assigns to each $U_j$ the properties we want to consider on it (i.e. some properties that it satisfies and which we want to consider on it; see the discussion at the beginning of this subsection). This is done in the same way as for the operations: consider a tuple of (equivalence classes of) encoding functions, say $(f_1, ..., f_n)$. If $U_p$ satisfies these properties, and we want to consider on it only these ones, then we assign $U_j$ the tuple  $(f_1, ..., f_n)$. More precisely, we assign $U_j$ an equivalence class, since we consider again two tuples containing the same (equivalence classes of) encoding functions as equivalent. Call this equivalence relation $\sim_2$. Then the map
$$g_2 : \mathcal{U} \rightarrow \mathcal{T}_2$$
defined by:
$$ U_p \xmapsto{g_2} (f_1, ...) $$
precisely describes the properties we want to consider on $U_p$. Here $\mathcal{T}_2$ is the collection of all the (equivalence classes of) tuples of properties that are associated to at least one of the fixed $U_j$'s in $X$. This implies that $g_2$ is surjective.\\
\subsection{Non-algebraic structures}
Our last description involves non-algebraic structures (e.g. norms, metrics, manifolds, ... defined over the algebraic structure; of course, the algebraic structure will be still called this way: if we want, we can also say that it is a "mixed" algebraic structure, meaning that we have defined additional structures such as the ones above). It should be clear by now how we will proceed: we assign to each $U_j$ a certain element that completely describes the non-algebraic additional structure.\\
Consider, for instance, a metric defined over a certain algebraic structure $A$. We just need to assign to $A$ the metric function, say $d$. Another example is given by Lie groups: here we need to assign a topology to a group $G$ (this can also be different from the topology on $X$: as we have already said, we either drop the structure on $X$ temporary (and here we will do this) or we consider $G$ as a bitopological space, and then apply all the considerations for the Lie group to the topology we want to deal with), an atlas describing the manifold (and we also need to require the transition maps to be $C^{\infty}$) and a function describing the smoothness of the operations on the group. The first passage is simple: just let $\tau$ be any topology on $G$ such that $G$ is Hausdorff and second countable. Then consider an atlas $\lbrace (V_{\alpha}, \phi_{\alpha}) \rbrace$ containing charts such that (by definition of atlas) the $V_{\alpha}$'s cover $G$ and $\phi_{\alpha}$ are homeomorphisms of $V_{\alpha}$ onto an open subset of $\mathbb{R}^n$ (this can be easily generalised to complex manifolds, ...). We also require this atlas to be such that the transition maps
$$ \phi_{\beta} \circ \phi_{\alpha}^{-1}|_{\phi_{\alpha}(V_{\alpha} \cap V_{\beta})} : \phi_{\alpha}(V_{\alpha} \cap V_{\beta}) \rightarrow \phi_{\beta}(V_{\alpha} \cap V_{\beta})$$
are $C^{\infty}$. Then, as shown in Proposition 5.10 and the comment below it in [3], $G$ is a smooth manifold. In order to be a Lie group, we also need the map (here $\cdot$ is the operation defined on $G$)
$$ (x,y) \xmapsto{f_L} x \cdot y^{-1} $$
to be smooth from the product manifold $G \times G$ into $G$, so suppose it is smooth. We then assign to $G$ via $g_3$ the collection $\lbrace \tau, \lbrace (V_{\alpha}, \phi_{\alpha}) \rbrace , f_L \rbrace$, which completely describes the non-algebraic structure defined over the group (note that we assign $f_L$ only to specify that the operation defined by this function is smooth; we would not assign this map to $G$ if it were not smooth (or, more generally, if it did not satisfy a certain property that describes the additional non-algebraic structure), because it would not be useful).\\
If we do not add any non-algebraic structure to $A$, we assign it $\lbrace \emptyset \rbrace$.\\
By this examples and the previous subsections, it is then clear that we can define
$$g_3 : \mathcal{U} \rightarrow \mathcal{T}_3$$
via
$$ U_p \xmapsto{g_3} \lbrace ... \rbrace$$
where $\lbrace ... \rbrace$ is a collection encoding the non-algebraic additional structures on $U_p$, and $\mathcal{T}_3$ contains all the (equivalence classes of) such collections (the ones such that at least one of the fixed $U_p$'s is mapped to them).\\
\subsection{The structure map}
We can finally define the structure map of a structured space, which will also allow us to define in an unambiguous way what a structured space (and hence also an algebraic structure) is. The idea is to "sum up" the actions of $g_1$, $g_2$ and $g_3$. To do this, we assign to each $U_p$ the collection $\lbrace g_1(U_p), g_2(U_p), g_3(U_p) \rbrace$, that is $\lbrace (\cdot_1, ...), (f_1, ...), \lbrace ... \rbrace \rbrace$. We denote by $\mathcal{T}$ the family of all these collections (the ones such that at least one of the fixed $U_j$'s is mapped to them). Thus, we have the following:
\begin{definition}\label{Def:2.2}
The structure map of a structured space $X$ precisely describes all its local structures, i.e. all the fixed $U_p$'s with their fixed algebraic structures. The structure map
$$f_s:\mathcal{U} \rightarrow \mathcal{T}$$
is defined as follows:
\begin{equation}\label{Eq:2.2}
U_p \xmapsto{f_s} \lbrace g_1(U_p), g_2(U_p), g_3(U_p) \rbrace
\end{equation}
\end{definition}
We note that, by definition of $\mathcal{T}$, $f_s$ is surjective. This map constitutes a fundamental result in the theory of structured spaces, because it precisely describes their local structures. In fact, we can define (this time, rigorously and unambiguously) what a structured space is as follows:
\begin{definition}[Structured space; final version]\label{Def:2.3}
A structured space is a couple $(X,f_s)$, where $X$ is a topological space and $f_s$ is the structure map that assigns to each subset of $X$ contained in the domain $\mathcal{U}$ its local structure.
\end{definition}
Notice that $f_s$ specifies both the fixed neighborhoods $U_p$'s (because its domain contains them) and also their fixed algebraic structures.\\
As a particular case, we can finally define precisely what an algebraic structure is:
\begin{definition}[Definition of algebraic structure]\label{Def:2.4}
An algebraic structure is a couple $(X,f_s)$, where $X$ is a set and $f_s$ is the structure map that assigns to $X$ the operations, the properties and the additional non-algebraic structures that we want to consider on it.
\end{definition}
We note that in Definition \ref{Def:2.4} we have not required $X$ to be a topological space, since that assumption is used to define the neighborhoods of a structured space, and here this is clearly not necessary. This way, we have a general definition of algebraic structure (without considering a topology on it). We also notice that, via the structure map, we can precisely describe also algebraic structures that are different from the "usual" ones, i.e. the ones considered in many contexts, such as rings, groups, C$^{*}$ algebras, ... For an example of a different algebraic structure, see for instance the structure map $f_{1,2}$ in Proposition \ref{Prop:3.1}.\\
We now give an example of construction of a structured space starting from a collection $\mathcal{U}$ of sets and a structure map $f_s : \mathcal{U} \rightarrow \mathcal{T}$. This construction can be performed in general, as will be clear by the example, so we can state the following:
\begin{theorem}[Construction of structured spaces]\label{Thm:2.1}
Consider a (possibly uncountable) collection of sets $\mathcal{U}:=\lbrace A_p \rbrace$ (each of which with more than one element). Define a priori (if possible) a structure map $f_s : \mathcal{U} \rightarrow \mathcal{T}$ that assigns to each element of $\mathcal{U}$ the structure we want to consider on it. Then, the space
$$X:= \bigcup_p A_p$$
endowed with the topology $\tau$ constructed (after a simple generalisation) in Example \ref{Ex:1.1} is a structured space.
\end{theorem}
\begin{example}\label{Ex:2.2}
\normalfont Consider four sets $U_i$ with more than one element each, and let $f_s$ be an a priori structure map defined as follows:\\
$$ U_1 \xmapsto{f_s} \lbrace (+_1, \cdot_1, \cdot_2), (g_1, g_2, g_3), \lbrace \emptyset \rbrace \rbrace$$
$$ U_2 \xmapsto{f_s} \lbrace (\cdot_3), (f_1, f_3), \lbrace \emptyset \rbrace \rbrace$$
$$ U_3 \xmapsto{f_s} \lbrace (+_2), (f_1, f_2, f_3, f_4, f_5), \lbrace \emptyset \rbrace \rbrace$$
$$ U_4 \xmapsto{f_s} \lbrace (\cdot_4), (f_1, f_2, f_3, f_4), \lbrace \tau, \lbrace (V_{\alpha}, \phi_{\alpha}) \rbrace , f_L \rbrace \rbrace$$
where:\\
$\bullet$ the algebraic structure assigned to $U_1$ has three operations and satisfies only the properties of closure w.r.t. its three operations (the $g_i$'s encode the closure w.r.t. these operations, one for each $g_i$). We do not add any additional non-algebraic structure;\\
$\bullet$ $f_1$, ..., $f_5$ are the encoding functions w.r.t. the unique operation defined on each $U_i$ ($i \neq 1$) of the following properties, respectively: closure, associativity, identity, invertibility, commutativity. As noticed also in Remark \ref{Rm:2.1}, the encoding functions themselves do not depend on a fixed structure, but when we consider them we first assign them the structure as domain, and then verify if it is $\equiv 0$ over it. If it is, then the property is satisfied; otherwise, it is not, and so we do not associate that encoding function to the structure. Here, with a little abuse of notation, we have indicated all the properties with respect to (generally) different operations in the same way, even though they are different; however, this is not confusing here. Moreover, $\tau$ is a topology on $U_4$ such that $U_4$ is Hausdorff and second countable, $\lbrace (V_{\alpha}, \phi_{\alpha}) \rbrace$ is an atlas whose transition maps are $C^{\infty}$, and $f_L$ is the smooth map $(x,y) \xmapsto{f_L} x \cdot_4 y^{-1}$;\\
$\bullet$ following what written above, we can say that $U_2$ is a unital magma, $U_3$ is an abelian ("additive") group, and $U_4$ is a Lie group.\\
The space
$$X:= \bigcup_{i=1}^{4} U_i$$
endowed with the topology $\tau$ constructed as in Example \ref{Ex:1.1} is, by Proposition \ref{Prop:1.1}, a structured space (more precisely, $(X,f_s)$ is a structured space, and $f_s$ is indeed a structure map).\\
This can be clearly generalised, obtaining Theorem \ref{Thm:2.1}.
\end{example}
\subsection{A modified structure map}
As said at the beginning of this section, we also have the possiblity to assign to each $p \in X$ the corresponding algebraic structure of the fixed $U_p$ (this fact has been pointed out by Helge Gl\"ockner in a personal communication). This can be useful in particular when $U_p = U_q$ for $p \neq q$. We can thus construct a "modified" structure map which also suites this need. We define the modified structure map $\widehat{f}_s$ as follows:
\begin{equation}\label{Eq:2.3}
p \xmapsto{\widehat{f}_s} f_s(U_p)
\end{equation}
Hence, this modified structure map represents another way to describe precisely a structured space: indeed, if instead of considering all the fixed structures we want to take the point of view of $X$ itself, we can describe the local structures via
$$\widehat{f}_s : X \rightarrow \mathcal{T}$$
defined above.\\
Actually, we notice that $f_s$ and $\widehat{f}_s$ are "equivalent", in the sense that we can use any of them in the definition of structured space. Indeed, if we know $f_s$ we can obviously obtain $\widehat{f}_s$ (by the definition above), while it is also clear that if we have $\widehat{f}_s$ we can obtain $f_s$, because to each $p\in X$ we assign a fixed $U_p$. Thus, using $f_s$ or $\widehat{f}_s$ in the definition of structured space (and algebraic structure) does not make any difference.\\
We also define a fundamental equivalence relation, '$\equiv$', which gives the usual meaning of 'equivalent algebraic structure' to $f_s$ and to $\widehat{f}_s$. We say that $f_s(A) \equiv f_s(B)$ (and then it clearly follows a definition of '$\equiv$' for $\widehat{f}_s$) if the following three conditions are satisfied:\\
(i) they have the same number of operations (which can also be different)\\
(ii) they have the same additional non-algebraic structure\\
(iii) there is a bijection between the operations given by $g_1(A)$ and the ones given by $g_1(B)$ such that, whenever $+_t$ is an operation defined on $A$ which satisfies some properties (given by encoding functions in $g_2(A)$), then the corresponding (via the said bijection) operation $\cdot_r$ on $B$ satisfies precisely the same properties (for instance, if $+_t$ is commutative, then $\cdot_t$ is commutative, ...), which are given by encoding functions in $g_2(B)$.\\
(i) should be clear since, for instance, we are used to groups (consider any such group, say $A$) under $+$, or under some other operation (consider any such group, say $B$) $\cdot$, ... They are all groups (even though under different operations), so they have the same algebraic structure (i.e. their algebraic structures are equivalent: $f_s(A) \equiv f_s(B)$); the fact that they have different operations follows from: $f_s(A) \neq f_s(B)$. However, since we will not need to use $=$ in the usual sense, we will occasionally use it while actually meaning $\equiv$ (we will however avoid this in most cases). The other points should also be clear, for similar reasons.\\
We notice that, if for instance $A$ is consider a group under $f_s$ and $B$ is considered a magma under this structure map (even though it could also be considered as a group), then their structures are different (i.e. $\not \equiv$), because (ii) above is not satisfied. Thus, we always need to pay attention to the fact that algebraic structures can also be considered as their substructures, but this will clearly modify the structure map $f_s$, which has various consequences on the equivalences '$\equiv$', as in the example above.
\section{Some ways to construct structured spaces}
In this section we construct some structured spaces starting from other structured spaces; for example, we consider products of structured spaces, isomorphic/homomorphic/... spaces, ... We will show that these kinds of spaces are structured spaces in various propositions, and then we will summarise everything in a Theorem at the end of this Section. We will also consider a classification of structured spaces starting from the classifications of the fixed local structures contained in it.\\
We start with product spaces:
\begin{proposition}[Product space]\label{Prop:3.1}
Let $(X,f_1)$, $(Y,f_2)$ be two structured spaces. Then $(X \times Y, f_{1,2})$ (endowed with the product topology) is a structured space.
\end{proposition}
\begin{proof}
\normalfont Endow $X \times Y$ with the product topology. We then only need to show that every $(x,y) \in X \times Y$ has a neighborhood which is an algebraic structure. Consider the neighborhoods $U_p \times V_t$ for some $U_p \in \mathcal{U}$, $V_t \in \mathcal{V}$. We define an algebraic structure on $X \times Y$ component-wise as follows:
$$ (x_1,y_1) \cdot_{\times, (1,1)} (x_2,y_2):=(x_1 \cdot_{X, 1} y_1, x_2 \cdot_{X, 1} y_2) $$
and so on. Notice that $\cdot_{X, i}$ is the $i$-th (clearly, it does not matter how we order the operations) operation on $X$, and $\cdot_{Y, i}$ is the $i$-th operation on $Y$. Then, the space $X \times Y$ will have (if $X$ has $n$ (different) operations and $Y$ has $m$ (different) operations) $nm$ operations. Hence, the function $g_1$ in the structure map is determined this way.\\
Now suppose that the properties of some $U_p$ are $(h_1, ..., h_j)$ and of $V_t$ are $(z_1, ..., z_i)$. Then we define the following properties on $U_p \times V_t$: the property $(h_1,z_1)$, which are just the properties $h_1$, $z_1$ w.r.t. the first component of the couple, the second component, respectively; ...; the property, in general, $(h_r,z_d)$, which means that $U_p \times V_t$ satisfies the property $h_r$ on the first component and $z_d$ on the second component. The fact that we consider two properties is not against what we said in Section 2, because here we are considering them in a couple in a product space (so the property is a couple, and not the two separated properties). This way, we obtain all these properties from the ones in each factor space, and hence also $g_2$ of the structure map is determined.\\
It only remains to consider the additional non-algebraic structures. In this case, we can just consider the (formal) product of collections $g_3(U_p) \times g_3(V_t)$ for every $U_p \in \mathcal{U}$, and for every $V_t \in \mathcal{V}$. This means that the new structure $U_p \times V_t$ will have a product mixed structure; for instance, if $U_p$ is a Lie group and $V_t$ is a Banach space, then $U_p \times V_p$ is a new mixed structure that reduces, under a canonical projection, to a Lie group or a Banach space, depending on the projection considered ($\pi_1$ if w.r.t. the first component, ...). It is sometimes possible that this product space is a "known" structure (e.g. the product of two groups is still a group), but (via the structure map) we can also define these product mixed structures in any situation, instead of only particular ones.\\
The product structure map on $X \times Y$ constructed this way will be called $f_{1,2}$. It is then clear that $(X,f_{1,2})$ is a structured space.
\end{proof}
We proceed with isomorphisms/homomorphisms/... (here, with homomorphisms we mean a map between two sets endowed with the same algebraic structures such that the operations are preserved (for instance, if $f$ is a homomorphism, then $f(x \times y)=f(x) \times f(y)$ for the operation '$\times$' defined on both structures)). The definition of isomorphism between structured spaces is defined locally, as we could expect. Suppose that $\mathcal{U}$ and $\mathcal{V}$ are the collections of fixed local structures of the structured spaces $X$, $Y$, respectively. Assume also that there exists at least one map $h: \mathcal{U} \rightarrow \mathcal{V}$ that assigns to each element in $\mathcal{U}$ an element with the same algebraic structure in $\mathcal{V}$. After some reordering, we can assume that $U_p$ is mapped to $V_p$ by $h$. Then, if $U_p$ and $V_p$ are isomorphic (where the isomorphism is defined accordingly to the (same) algebraic structure of these two neighborhoods) for every $U_p$, $h(U_p)=V_p$, then we say that $X$ and $Y$ are isomorphic as structured spaces. If instead of isomorphisms we had homomorphisms/..., we would have homomorphic/... structured spaces.\\
We prove the following result:
\begin{proposition}\label{Prop:3.2}
Let $(X,f)$ be a structured space. Let $\mathcal{U}:=\lbrace U_p \rbrace$ be, as usual, the collection of all the fixed local structures of $X$. For each $U_p \in \mathcal{U}$, consider any set $V_p$ which is isomorphic/homomorphic/... to $U_p$ (where we consider on $V_p$ the same algebraic structure of the corresponding $U_p$). Let:
$$\widehat{X}:= \bigcup_p V_p$$
Then $(\widehat{X},f)$ is a structured space and it is isomorphic/homomorphic/... to $(X,f)$.
\end{proposition}
\begin{proof}
Since every $V_p$ is still an algebraic structure, and actually it is exactly the same structure of $U_p$, we conclude that $(\widehat{X},f)$ (we consider the same structure map of $X$, thanks to the relation given by isomorphism/homomorphism/...) is a structured space (see Proposition \ref{Prop:1.1}). Furthermore, by the definition above of isomorphism/homomorphism/... between structured spaces, it is clear that $X$ and $\widehat{X}$ are isomorphic/homomorphic/...
\end{proof}
We now proceed considering quotient structured spaces. In order to do this, we need to define what we shall call 'structured ideal spaces'. Consider (we suppose that this can be done for every element of $\mathcal{U}$), for every fixed $U_p$ in a structured space $X$, a certain structure, say $\widehat{U}_p$, such that we can define the quotient space $U_p / \widehat{U}_p$ (for instance, if $U_p$ is a group, we can consider any normal subgroup of $U_p$, while if $U_p$ is a ring we consider an ideal $I$ (from this we take the name of these spaces)). Then, the space:
$$\widehat{X}:= \bigcup \widehat{U}_p$$
is called a structured ideal space of $X$. It is then clear that we can define the space:
\begin{equation}\label{Eq:3.1}
X/ \widehat{X}:= \bigcup U_p / \widehat{U}_p
\end{equation}
which is called the quotient space of $X$ by its ideal $\widehat{X}$. We note that the structures $\widehat{U}_p$ are such that the quotient spaces $U_p / \widehat{U}_p$ have still the same structure of $U_p$ (this remark is important in particular for "not well known" algebraic structures: while for groups, rings, ... we are used to consider quotient spaces, there are less known structures that need some definition for quotients, and this is done accordingly to this condition). We prove the following:
\begin{proposition}[Quotient space]\label{Prop:3.3}
Let $(X,f)$ be a structured space and let $\widehat{X}$ be a structured ideal of $X$. Define $X / \widehat{X}$ as in \eqref{Eq:3.1}. Then, defining a topology on $X / \widehat{X}$ by using Proposition \ref{Prop:1.1}, the space $(X / \widehat{X}, f)$ is a structured space.
\end{proposition}
\begin{proof}
\normalfont Since, by the assumption discussed above, all the quotient spaces $U_p / \widehat{U}_p$ are defined in such a way that they still have the same structure $U_p$, we can use the same structure map $f$ that assigns to each $U_p / \widehat{U}_p$ its local structure, which we consider to be the same as the structure of $U_p$ thanks to the relation given by the quotients. It is then clear that, since we define a topology on the quotient space using Proposition \ref{Prop:1.1} (this can be done since all the fixed local structures of $X / \widehat{X}$ are algebraic structures, as said above), this space is a structured space: just consider the $U_p / \widehat{U}_p$'s as fixed neighborhoods.
\end{proof}
Before studying classifications of structured spaces, we show that also direct limits give rise to structured spaces. We first recall (see also [11-13]) what a direct limit is (we do not use the general definition involving category theory here, because we do not need it). Let $(I, \leq)$ be a directed set and consider a family of sets, say $\lbrace A_i \rbrace$, endowed with the same algebraic structure; moreover, consider homomorphisms $f_{i,j} :A_i \rightarrow A_j$ for $i \leq j$ such that:\\
$\bullet$ $f_{i,i}$ is the identity of $A_i$;\\
$\bullet$ $f_{i,k}=f_{j,k} \circ f_{i,j}$ for all $i \leq j \leq k$.\\
Then, for the direct system (over $I$) $\langle A_i, f_{i,j} \rangle$, we define (if it exists) the direct limit:
\begin{equation}\label{Eq:3.2}
\varinjlim A_i := \sqcup_i A_i / \sim
\end{equation}
where $\sqcup$ denotes the disjoint union, and $x_i \sim x_j$ iff there is some $k$ with $i \leq k$ and $j \leq k$ such that $f_{i,k}(x_i)=f_{j,k}(x_j)$. Furthermore, we define the canonical functions $\phi_i : A_i \rightarrow \varinjlim A_i$ as the maps that send each element to its equivalence class. We define algebraic operations on $\varinjlim A_i$ so that these $\phi_i$'s become homomorphisms. This way, we assume that $\varinjlim A_i$ has the same algebraic structure of all the $A_i$'s. We now prove:
\begin{proposition}[Direct limit]\label{Prop:3.4}
Consider collections $\mathcal{C}_i:=\lbrace _i U_p, p \in I_i \rbrace$, where $I_i$ (for fixed $i$) is an ordered set, and all the element of each collection are endowed with the same algebraic structure. Suppose that for each $i$ there exist the direct limits as defined by \eqref{Eq:3.2} (if there are several different direct limits for the same collection, choose one of them), and assume (using the homomorphisms $_i \phi_p$, as discussed above) that all these direct limits have the same algebraic structure of the corresponding direct system. Then the space
$$\varinjlim X := \bigcup_i \varinjlim \, _i U_p$$
is a structured space.
\end{proposition}
\begin{proof}
\normalfont We can fix the neighborhoods $\varinjlim \, _i U_p$ and use Proposition \ref{Prop:1.1}, together with the local structures defined on each fixed neighborhood, to conclude that $X$ is a structured space w.r.t. the structure map that assigns to each $\varinjlim \, _i U_p$ the local structure of every $_i U_p \in \mathcal{C}_i$.
\end{proof}
We can summarise briefly (see the previous propositions for the precise statements) what said up to now in the following:
\begin{theorem}\label{Thm:3.1} The following statements holds true:\\
(i) Finite products of structured spaces are structured spaces;\\
(ii) If for every local structure of a structured space we consider an isomorphic/homomorphic/... structure, then the set given by the union of all these structures is a structured space (and it is isomorphic/homomorphic/... to the former space);\\
(iii) The quotient of a structured space by a structured ideal space is a structured space;\\
(iv) Unions of direct limits form structured spaces.
\end{theorem}
We conclude this section with an important question which arises when studying algebraic structures: is it possible to classify in some way these structures? In particular, here we mean that a certain kind of structure is such that we can find a list of sets endowed with the same structure satisfing the following properties:\\
(i) every space endowed with that kind of algebraic structure is isomorphic to one and only one space on the list\\
(ii) there is no space on the list that is isomorphic to another set on the list.\\
We will call such a list a 'classification' of that kind of algebraic structure. There are many possible examples in this context. We just recall a few:\\
1) finitely generated abelian groups (which are all isomorphic, by the fundamental theorem of finitely generated abelian groups, to a group of the form $\mathbb{Z}^n \oplus \mathbb{Z}/k_1 \mathbb{Z} \oplus ... \oplus \mathbb{Z}/k_r \mathbb{Z}$ ($n$ is the rank, and the $k_i$'s are powers of not necessarily distinct prime numbers); see also [19] for more details on the theory of abelian groups);\\
2) Finite simple groups (the list containing cyclic groups, sporadic groups, ...);\\
3) Finite dimensional C$^*$-algebras (if $A$ is a C$^*$-algebra, then it is isomorphic to $\oplus_ {e \in \min A} Ae$, where $\min A$ indicates the set of minimal nonzero self-adjoint central projections of $A$);\\
... (see also the reference (in particular, [20-36]) for some papers on classification of various algebraic structures). So, it is natural to ask whether we can classify also structured spaces. By now, it should be clear how we will answer this question:
\begin{theorem}[Classification of structured spaces] \label{Thm:3.2}
Suppose that the fixed neighborhoods $U_p$ of a structured space $(X,f)$ can all be classified. Then the space $X$ can also be classified considering all the structured spaces isomorphic to it (where we consider the corresponding isomorphic $\widehat{U}_p$'s to be in the classification list of the respective local structures).
\end{theorem}
\begin{proof}
\normalfont This easily follows from Proposition \ref{Prop:3.2}. See also the example below.
\end{proof}
\begin{example}\label{Ex:3.1}
\normalfont Consider a structured space $(X,f)$ whose fixed local structures are only finite dimensional C$^*$-algebras and finitely generated abelian groups. Then, by the isomorphisms above, we know that:
$$ X \cong \bigcup \oplus_ {e \in \min A} Ae \cup \bigcup(\mathbb{Z}^n \oplus \mathbb{Z}/k_1 \mathbb{Z} \oplus ... \oplus \mathbb{Z}/k_r \mathbb{Z})$$
(where $\cong$ means isomorphic here), i.e. the space is isomorphic to the structured space whose local structures are some of the ones in the corresponding lists.
\end{example}
\section{Some measure-theoretic properties of structured spaces}
In this section we study some properties of some particular kinds of structured spaces. We first introduce an important and natural concept:
\begin{definition}\label{Def:4.1}
A structured space $(X,f)$ is partitionable if $U_p \cap U_j = \emptyset$ whenever $U_j \neq U_p$, with $U_p, U_j \in \mathcal{U}$, and with $\mathcal{U}$ supposed to be a countable collection of sets.
\end{definition}
It is clear that the collection $\mathcal{U}$ of a partitionable space $(X,f)$ is then a partition of the set $X$. We will consider various kinds of spaces related to partitionable ones. First of all, we need a bit of preparation.\\
Since a structured space $(X,f)$ is also a topological space, we can define the Borel $\sigma$-algebra $\mathcal{B}(X)$ on it, and then we can consider a (positive) measure $\mu: \mathcal{B}(X) \rightarrow [0,+\infty]$. When we deal with measure theory in relation to structured spaces, if not otherwise specified, we always tacitly assume that $X$ is a measure space w.r.t. $\mathcal{B}(X)$ and with a positive measure on it (not necessarily finite). We will also assume that all the fixed neighborhoods of $X$ belong to $\mathcal{B}(X)$ (using the topology in Proposition \ref{Prop:1.1}, they belong to the Borel $\sigma$-algebra by its definition, since they are open sets in $\tau$).\\
Now, it is clear that a partitionable space is such that:
$$\mu(X)=\sum_{i=1}^{k} \mu(U_i)$$
(by assumption, $\mathcal{U}$ is countable, so it has a certain number $k$ of elements, with $k$ possibly equal to $+\infty$) because the union of the $U_i$'s is $X$ and they are all disjoint. However, the converse is not true in general: we could have $\mathcal{U}$ with neighborhoods whose intersections have zero measure under $\mu$. This suggests the following:
\begin{definition}\label{Def:4.2}
A structured space $(X,f)$ is called $\mu$-local almost partitionable (or, more briefly, $\mu$-LA partitionable) if there exist a countable subcollection $\mathcal{C}$ of $\mathcal{U}$, say $\mathcal{C}=\lbrace U_i \rbrace_{i=1}^{k}$ (with $k$ possibly equal to $+\infty$), and a subset $A$ of $X$ (not necessarily an algebraic structure/structured space), with $\mu(A)=0$ (so we also assume $A \in \mathcal{B}(X)$), such that:\\
1) All the $U_i$'s in $\mathcal{C}$ are pairwise disjoint: $U_i \cap U_j = \emptyset$ whenever $i \neq j$;\\
2) $X$ can be written as:
\begin{equation}\label{Eq:4.1}
X= A \cup \bigcup_{i=1}^{k} U_i
\end{equation}
3) The following holds:
\begin{equation}\label{Eq:4.2}
\mu(X)=\sum_{i=1}^{k} \mu(U_i)
\end{equation}
(where, using Proposition \ref{Prop:1.1}, we assume that $U_i \in \mathcal{B}(X)$ $\forall i$).
\end{definition}
Partitionable spaces are always $\mu$-LA partitionable, whichever is the measure $\mu$ on $\mathcal{B}(X)$. It is clear that, if 1) and 2) above are satisfied, and if at least one of the $U_i$'s in $\mathcal{C}$ has infinite measure under $\mu$, then $X$ is $\mu$-LA partitionable, since \eqref{Eq:4.2} certainly holds.\\
Before starting to study some properties of these spaces, we prove the following simple:
\begin{lemma}\label{Lm:4.1}
Suppose $(X,f)$ is a structured space, say with $\mathcal{U}=\lbrace U_p \rbrace$. Consider any subcollection $\widehat{\mathcal{U}}$ of $\mathcal{U}$. Then, defining:
$$\widehat{X}:=\bigcup_{U_t \in \widehat{\mathcal{U}}} U_t$$
we have that $(\widehat{X},f|_{\widehat{\mathcal{U}}})$ is a structured space.
\end{lemma}
\begin{proof}
Since $\widehat{X} \subseteq X$, we can use the subspace topology. Then, it is clear by Proposition \ref{Prop:1.1} that $\widehat{X}$ is a structured space w.r.t. $f|_{\widehat{\mathcal{U}}}$.
\end{proof}
Thus, the union of all the $U_i$'s in Definition \ref{Def:4.2} is a structured space. However, as we have already noticed, $A$ is not required to be (in general) a structured space. Consider $Y:=X \setminus \bigcup_{i=1}^{k} U_i \subseteq A$ (clearly, $\mu(Y)=0$, since $Y=(\bigcup_{i=1}^{k} U_i)^c \in \mathcal{B}(X)$ (by definition of $\sigma$-algebra) and $Y \subseteq A$, with $\mu(A)=0$). We would like to add some points to $Y$ (and hence to $X$) so that $Y$ becomes an algebraic structure. We also want this extension to have zero measure under some natural extension of $\mu$. Therefore, we define:
\begin{definition}\label{Def:4.3}
Let $Y$ be as above. Suppose that $Z$ is a set such that:\\
(i) $Z \cap X = \emptyset$\\
(ii) $Z \cup Y$ is an algebraic structure\\
(iii) Letting $\widehat{\tau}$ be the extension topology defined over $Z \cup X$, suppose that $\mu$ can be extended to $\widehat{\mu}:\mathcal{B}(X \cup Z) \rightarrow [0,+\infty]$ (i.e. $\mu \equiv \widehat{\mu}$ over $\mathcal{B}(X) \subseteq \mathcal{B}(X \cup Z)$)\\
Then we say that $Z$ is a $\mu$-null addition to $X$.
\end{definition}
Let $S$ be a topological space (say, with topology $\gamma$) and let $P$ be a set which is disjoint from $S$. We recall that the extension topology is defined as the topology whose open sets are of the form $K \cup Q$, where $K \in \gamma$ and $Q$ is a subset of $P$.\\
We can now prove the following interesting extension result:
\begin{proposition}\label{Prop:4.1}
If $Z$ is a $\mu$-null addition to a $\mu$-LA partitionable space $(X,f)$, say w.r.t. a $\mu$-LA collection $\mathcal{C}=\lbrace U_i \rbrace_{i=1}^{k}$, then $\mathcal{C} \cup \lbrace Y \cup Z \rbrace$ is a partition for $X \cup Z$, and thus $X \cup Z$ is partitionable.
\end{proposition}
\begin{proof}
By the definition of $Y$, we have that:
$$ Y \cup Z \cup \bigcup_{i=1}^{k} U_i = (X \setminus \bigcup_{i=1}^{k} U_i) \cup Z \cup \bigcup_{i=1}^{k} U_i = X \cup Z$$
By (i) in Definition \ref{Def:4.3} and by the pairwise disjointness of the $U_i$'s, we thus conclude that the collection $\mathcal{C} \cup \lbrace Y \cup Z \rbrace$ is a partition of the \emph{set} $X \cup Z$. Since all the $U_i$'s are algebraic structure, and also $Y \cup Z$ is (by (ii) in \ref{Def:4.3}) an algebraic structure, we can conclude that $\mathcal{C} \cup \lbrace Y \cup Z \rbrace$ is a partition for the structured space (under the extension topology) $X \cup Z$.
\end{proof}
This result says that a $\mu$-LA partitionable space can always (if there exists at least one $\mu$-null addition) be extended to a partitionable space, in such a way that the extension has still the same measure under $\widehat{\mu}$. Thus, under some mild conditions, the theory of $\mu$-LA partitionable spaces can be reduced to the theory of partitionable spaces.\\
Another interesting kind of space which is similar to $\mu$-LA partitionable ones is defined as follows. Let $(X,f)$ be a structured space and consider a subcollection $\mathcal{C}:=\lbrace U_p \rbrace \subseteq \mathcal{U}$. Suppose that, letting $U_p \cap U_j=: A_{p,j}$ (from which we immediately note that $A_{p,j} \equiv A_{j,p}$), we have
$$ \mu(A_{p,j})=0 $$
for every $U_p \neq U_j$ in $\mathcal{C}$. Then, if:
$$X=\bigcup_{U_p \in \mathcal{C}} U_p$$
we say that $X$ is a $\mu$-union of $\mathcal{C}$. Notice that, if $X$ is a $\mu$-union of a countable $\mathcal{C}$, then \eqref{Eq:4.2} holds.\\
For example, if $\mathcal{P}$ is any partition of a partitionable space $(X,f)$, then $X$ is a $\mu$-union of $\mathcal{P}$, whichever is $\mu: \mathcal{B}(X) \rightarrow [0,+\infty]$.\\
We now additionaly suppose that the $U_p$'s in $\mathcal{C}$ have all the local structures of $X$. This means that $f|_{\mathcal{C}} \equiv f$. We call such a collection $\mathcal{C}$ a $\mu$-complete restriction (more briefly, $\mu$-CR) for the structured space $(X,f)$. Henceforth, we will consider $\mu$-CDR ($\mu$-completely distinguished partitions), which are $\mu$-CR such that all the $U_p$'s in $\mathcal{C}$ all have different structures (thus, $\mathcal{C}$ is composed by neighborhoods in $\mathcal{U}$ whose union is $X$, and such that they have all the local structures of $X$; furthermore, the structures of any $U_p$ is different from the structures of all the other $U_j$'s in $\mathcal{C}$). It is clear that $\mathcal{U} \setminus \mathcal{C}$ is "not so important" to $X$, because we can form $X$ only using the $U_p$'s in $\mathcal{C}$, and furthermore we have all the local structures of $X$ contained in $\mathcal{C}$. Indeed, we can prove that (here, '$\setminus$' with collections of sets work precisely in the same way as for sets: if $\mathcal{A}=\lbrace A, B \rbrace$ and $\mathcal{B}=\lbrace B \rbrace$ for some sets $A,B$, then $\mathcal{A} \setminus \mathcal{B}=\lbrace A \rbrace$):
\begin{proposition}\label{Prop:4.2}
Let $\mathcal{C}$ be a $\mu$-CDR for the structured space $(X,f)$. Then, if $U_p \in \mathcal{U} \setminus \mathcal{C}$ is the fixed neighborhood of $p$, we can replace it with an extension (supposing there exist $\mu$-null maintaining additions (i.e. the extended space maintain the same structure) contained in $X$ and containing $p$, for all the $U_j \in \mathcal{C}$ and for all the necessary $p \in X$) of the unique element of $\mathcal{C}$ with the same local structure, where the extension has the same measure of the neighborhood itself (so, this extension is a neighborhood of $p$ with an algebraic structure). The new collection containing also these extensions is not, in general, $\mu$-CRD, nor even a $\mu$-union for $X$.
\end{proposition}
\begin{proof}
\normalfont Since
$$X=\bigcup_{U_p \in \mathcal{C}} U_p$$
any $j\in X$ such that $U_j \in \mathcal{U} \setminus \mathcal{C}$ must belong to one of the $U_p$'s in $\mathcal{C}$. Then, we can think of using one of the $U_p \in \mathcal{C}$ instead of $U_j$. We have two cases: if $j$ belongs to the $U_p \in \mathcal{C}$ with the same structure of $U_j$, we are done. Otherwise, by assumption we can extend the $U_p \in \mathcal{C}$ with the same structure of $U_j$ using a $\mu$-null maintaining addition $Z_j$ (we note that the possibility to do this is due the fact that $j$ is not "too far in measure" from $U_p$, beacuse indeed $\mu(Z_j)=0$, so the extension is not "large" in measure). Since, by definition of $\mu$-null maintaining addition, $U_p \cup Z_j$ is an algebraic structure, and in particular it is the same algebraic structure of $U_p$ (think about, for instance, extension fields: we extend an algebraic structure adding some points and maintaining the same structure. For example, $\mathbb{Q}(\sqrt{2}):=\lbrace a + b \sqrt{2}, a,b \in \mathbb{Q} \rbrace$ is an extension field of $\mathbb{Q}$. Furthermore, if we consider the measure space $\mathbb{R}$ with Lebesgue measure $\lambda$, we have that the "added part" $Z=\mathbb{Q}(\sqrt{2}) \setminus \mathbb{Q}$ has measure $0$ under $\lambda$. Thus, $Z=\mathbb{Q}(\sqrt{2}) \setminus \mathbb{Q}$  is an example of $\mu$-null maintaining addition), and furthermore this extension is contained in $X$ (always by hypothesis), we can consider the collection $\widehat{\mathcal{C}}$ defined by adding to $\mathcal{C}$ all these extensions.\\
It is clear that the intersections of the elements of $\widehat{\mathcal{C}}$ do not necessarily have measure $0$ under $\mu$, so $X$ is not, in general, a $\mu$-union of $\widehat{\mathcal{C}}$.
\end{proof}
The new collection obtained is, in some sense, more "natural" than $\mathcal{U}$, since it does not consider "too many" fixed neighborhoods as the original one. It only contains the "necessary" ones, with the addition of some extensions which are "almost the same" (in measure) of the previous ones. This collection $\mathcal{\widehat{C}}$ will be called the $\mu$-essential part of $\mathcal{U}$.\\
We also note another interesting example. It can be useful to consider structured spaces having the following property: if $p$ is "near" $q$, then we would like to have $U_p$ with the same structure of $U_q$. The notion of distance needed can be achieved without necessarily using metrics: if we consider a certain measure $\mu$ on the structured space $X$, we can think that $p$ is $\mu$-near $q$ if $\mu(U_p \cap U_q)>0$. Then, we can define:
\begin{definition}\label{Def:4.4}
A structured space $X$ is called locally $\mu$-homogeneous if the following property is satisfied: whenever $U_p \neq U_q$ and $\mu(U_p \cap U_q) > 0$ (which also imply that the two neighborhoods are not disjoint), $U_p$ and $U_q$ have the same algebraic structure (i.e. $f_s(U_p) \equiv f_s(U_q)$, where '$\equiv$' has been defined in Section 2.5). A structured space is called globally $\mu$-homogeneous if the following stronger property is satisfied ($U_p \neq U_q$): $\mu(U_p \cap U_q)>0$ $\Leftrightarrow$ $f_s(U_p) \equiv f_s(U_q)$.
\end{definition}
This kind of space can be useful in applications, since it is sometimes more natural to consider the same kind of algebraic structure for points that are "near" in some sense (here, in the sense of $\mu$). We notice that, even if two sets are disjoint, their measure under $\mu$ can be $0$: this makes the above definition more interesting, since we do not only ask their intersection to be nonempty.
\begin{example}\label{Ex:4.1}
\normalfont Consider the structured space $X$ defined in Example \ref{Ex:1.2}. Since the fixed neighborhoods are the $GL(n, \mathbb{R})$'s in that case, it is clear that $X$ is locally $\mu$-homogeneous, whichever is the positive measure $\mu$. It is also clear that $X$ is not globally $\mu$-homogeneous.
\end{example}
\begin{example}\label{Ex:4.2}
\normalfont If a structured space $X$ is a $\mu$-union of $\mathcal{U}$, then it is clear that $X$ is locally $\mu$-homogeneous, since $\mu(U_p \cap U_q)=0$ for every $U_p \neq U_q$ by assumption. Thus, also partitionable spaces are locally $\mu$-homogeneous. It is also clear that they need not be globally $\mu$-homogeneous. However, if a structured space $X$ is $\mu$-CRD w.r.t. $\mathcal{U}$, then it is globally $\mu$-homogeneous.
\end{example}
We conclude this section with a discussion about the local structures of a structured space $X$, which will lead us to prove one of the most important results of this paper. Intuitively, a point $x$ which is contained in some fixed $U_p$'s in $\mathcal{U}$ "have" all the local structures of those $U_p$'s. This means that the point $x$ is used (for instance, relatively to operations) in all those structures. We can then think of distinguish the points which are contained (and hence "used") in more local structures than others. This can be done by assigning to each $x \in X$ the subcollection of $\mathcal{U}$ containing all the fixed neighborhoods that contain $x$:
\begin{equation}\label{Eq:4.3}
x \xmapsto{h} \lbrace U_p \in \mathcal{U} : x \in U_p \rbrace \subseteq \mathcal{U}
\end{equation}
with $h:X\rightarrow \mathcal{L}$, where $\mathcal{L}$ contains all the families of subcollections of $\mathcal{U}$ (the "power set", but with collections of sets) excluding the ones containing the empty set. Notice that:
\begin{equation}\label{Eq:4.4}
h(x)=\mathcal{U} \Leftrightarrow x \in \bigcap_{p\in X} U_p
\end{equation}
We now define a notion of order on $X$. The first attempt is to define '$\leq$' on $X$ using $h$:
$$x \leq y \Leftrightarrow h(x) \subseteq h(y)$$
However, even though this partial relation is clearly reflexive and transitive, it is not antisimmetric in general. We thus define the following equivalence relation $\sim$:
\begin{equation}\label{Eq:4.5}
x \sim y \Leftrightarrow h(x)=h(y)
\end{equation}
Then, $X / \sim$ is also antisymmetric and so it is a partially ordered set under $\leq$. One of the main result of this paper is that any structured space $X$ such that $h$ is surjective induces a lattice this way, and conversely every lattice induces a quotient $Y / \sim_1$ for some structured space $Y$ and some equivalence relation $\sim_1$ (we explicitely note that $\sim_1$ is generally not the same of $\sim$, and indeed in this converse statement we will not use (in the proof) the partial order $\leq$ defined above, but the one defined on the lattice by its definition), and this can be done in such a way that the function $\widehat{h}$ obtained is surjective. This is the statement of the next important:
\begin{theorem}[Structured spaces-lattices]\label{Thm:4.1}
Every structured space $X$ with surjective $h$ is such that $X / \sim$ is a lattice under some partial order $\leq$, where $\sim$ is an appropriate equivalence relation. Conversely, every lattice can be written as $Y / \sim_1$ for some equivalence relation $\sim_1$ and some structured space $Y$, and these can be chosen so that the corresponding $\widehat{h}$ is surjective.
\end{theorem}
\begin{proof}
\normalfont We prove the first part by constructing the equivalence relation $\sim$ and the partial order $\leq$, which are actually the ones defined above. We only need to verify that $X / \sim$ is a lattice under the order $\leq$. We start checking that it is a join semilattice. Consider any set $\lbrace a,b \rbrace$ with $a,b \in X$. We need to prove that $\lbrace a,b \rbrace$ is such that there exists a least upper bound (if it exists, it is well known that this must be unique), i.e., if we consider the set of all the upper bounds of $\lbrace a,b \rbrace$, which is $Z(a,b):=\lbrace x\in X : s \leq x, s\in \lbrace a, b \rbrace \rbrace$, we have to prove that there exists an element in this set, say $a \vee b$, such that $a \vee b \leq x$ whichever is the upper bound $x\in Z(a,b)$. By "unpacking" the definition of $\leq$, we notice that $Z(a,b)$ constists of the $x \in X$ that are contained at least in all the $U_p$'s in which $a$ is contained and also in all the $U_j$'s in which $b$ is contained. This means that we could think of the least upper bound as the collection $h(a) \cup h(b)$. However, it is not necessary (in general) that there is at least one $y$ in $X$ such that $h(y)=h(a) \cup h(b)$. The assumption of the Theorem on $h$ (surjectivity onto the "power collection" without the empty sets) implies that, whichever is the collection $\mathcal{C}$ in $\mathcal{L}$, there is certainly at least one $y \in X$ such that $h(y)= \mathcal{C}$. Thus, there always exists the least upper bound, which is $a \vee b = y$, with $y$ such that $h(y)=h(a) \cup h(b)$. The least upper bounds of these sets with two values are called joins. We conclude that $X/ \sim$ is a join semilattice under $\leq$. We notice that, without $\sim$, $\leq$ would not be in general a partial order and the least upper bound would not be necessarily unique.\\
Similarly, by the surjectivity of $h$, we can conclude that the greatest lower bound $a \wedge b = y$, with $y$ such that $h(y)=h(a) \cap h(b)$, always exists (notice that, by the surjectivity of $h$ and by \eqref{Eq:4.4}, $\bigcap_p U_p \neq \emptyset$. This implies that the intersection of $h(a)$ and $h(b)$ is never the empty set, and of course this intersection collection does not contain the empty set because $h(a)$ and $h(b)$ do not contain it by the definitions of $h$ and $\mathcal{L}$), so $X/ \sim$ is also a meet semilattice. Thus, $X/ \sim$ is a lattice.\\
To prove the second part, we construct some $Y$ and $\sim_1$ satisfing the said conditions. More precisely, we consider $Y$ to be $X$ itself, and $\sim_1$ to be simply equality: $x \sim_1 y \Leftrightarrow x=y$. It is then clear that $Y / \sim_1$ is in fact $X$, and so it is a lattice. $\sim_1$ is obviously an equivalence relation, and $Y$ is a lattice (because it is equal to $X$), and therefore it is an algebraic structure (the operations given by $\vee$ and $\wedge$) and hence it is a structured space. Thus, we only need to verify that the corresponding $\widehat{h}$ is surjective. But this is trivial, since $\widehat{h}$ only maps to $\lbrace X \rbrace$, and then is clearly surjective.
\end{proof}
\section{Connectedness}
We conclude this paper with some simple results on connectedness. We first define:
\begin{definition}\label{Def:5.1}
A structured space is completely open (closed) if all the fixed neighborhoods $U_p \in \mathcal{U}$ are open (closed). Here with open (closed) we also consider the possibility that they are clopen: the important thing is that they are at least open (closed).
\end{definition}
Before proceeding, we recall some definitions:
\begin{definition}\label{Def:5.2}
A topological space $X$ is \emph{connected} if it cannot be represented as the union of two or more disjoint non-empty open subsets. Otherwise, it is \emph{disconnected}.\\
A topological space $X$ is \emph{hyperconnected} if no two non-empty open sets are disjoint.\\
A topological space $X$ is \emph{ultraconnected} if no two non-empty closed sets are disjoint.
\end{definition}
We can now prove the following result:
\begin{proposition}\label{Prop:5.1}
The following statements hold true:\\
(i) Every partitionable space which is completely open is disconnected;\\
(ii) If a structured space is completely open (closed) and hyperconnected (ultraconnected, respectively), then it is not partitionable.
\end{proposition}
\begin{proof}
(i) By definition of partitionable space, the space is the union of all its $U_p$'s, which are open (and disjoint) by hypothesis. Thus, by Definition \ref{Def:5.2}, it must be disconnected.\\
(ii) Suppose a structured space is completely open and hyperconnected. Then the definition of hyperconnectedness together with the assumption that all the $U_p$'s are open imply that the fixed neighborhoods are not disjoint, so the space cannot be partitionable. The proof of the other statement is similar.
\end{proof}
This gives some criteria to check if a structured space is partitionable.
\section{Conclusion}
In this paper we have introduced a new kind of space which locally resembles some algebraic structures. We have precisely described the fixed local structures with the structure map, which makes the notions of structured space and algebraic structure unambiguous. We have then shown some examples of constructions of structured spaces starting from other structured spaces, and we have proceeded with some properties related to partitionable and $\mu$-almost partitionable spaces. We have finally proved one of the main result of this article, Theorem \ref{Thm:4.1}, which shows a really interesting and important relation between structured spaces and lattices. We have then concluded with some simple properties related to connectedness. We notice the similarity of the structure map (and the modified structure map) and of the map $h$ defined before Theorem \ref{Thm:4.1} with vector bundles and presheaves: in those cases, we attach to each point/open subset something (a vector space, or a set of sections and restriction morphisms, for instance), while here we can see the similar process of assigning to each fixed neighborhood $U_p$ (or to each $p \in X$) a certain algebraic structure via $f_s$ (or via $\widehat{f}_s$), and to each point of a decomposable space all of its local structures via $h$.\\
\\
\textbf{Acknowledgements}: I thank Helge Gl\"ockner for his useful comments on the first version of this paper. His suggestions helped me improve various parts of this article.\\
\\
\textbf{Conflict of Interest}: The author declares that he has no conflict of interest.\\
\\
\begin{large}
\textbf{References}
\end{large}
\\
$[1]$ Cohen, F.R.; Stafa, M. (2016). A Survey on Spaces of Homomorphisms to Lie Groups. In: Callegaro, F.; Cohen, F.; De Concini, C.; Feichtner, E.; Gaiffi, G.; Salvetti, M. (eds) Configuration Spaces. Springer INdAM Series, vol 14. Springer, Cham\\
$[2]$ De la Harpe, P. (2000). Topics in geometric group theory. Chicago lectures in mathematics. University of Chicago Press.\\
$[3]$ Tu, L. W. (2008). An Introduction to Manifolds, 2nd edition. Springer.\\
$[4]$ Cohn, D. L. (2013). Measure Theory, Second Edition. Birkh\"auser, New York, NY\\
$[5]$ Hungerford, T. W. (2012). Abstract Algebra: An Introduction, Third Edition. Cengage Learning\\
$[6]$ Kelly, J. C. (1963). Bitopological spaces. Proc. London Math. Soc., 13(3) 71-89\\
$[7]$ Sakai, S. (1971). C$^*$-algebras and W$^*$-algebras. Springer\\
$[8]$ Kempf, G. R. (1995). Algebraic Structures. Vieweg+Teubner Verlag\\
$[9]$ Birkhoff, G. (1967). Lattice Theory, 3rd ed. Vol. 25 of AMS Colloquium Publications. American Mathematical Society.\\
$[10]$ Davey, B. A.; Priestley, H. A. (2002). Introduction to Lattices and Order, Cambridge University Press\\
$[11]$ Gl\"ockner, H. (2007). Direct limits of infinite-dimensional Lie groups compared to
direct limits in related categories. J. Funct. Anal. 245, 19-61\\
$[12]$ Gl\"ockner, H. (2011). Direct limits of infinite-dimensional Lie groups, pp. 243-280 in:
Neeb, K. - H.; Pianzola, A. (Eds.). "Developments and Trends in Infinite Dimensional Lie Theory", Progr. Math. 288, Birkh\"auser, Boston.\\
$[13]$ Gl\"ockner, H. (2019). Direct limits of regular Lie groups. Preprint (arXiv:1902.06329)\\
$[14]$ Mac Lane, S. (1998). Categories for the Working Mathematician, Graduate Texts in Mathematics, 5 (2nd ed.), Springer-Verlag\\
$[15]$ Segal, I. (1947). Irreducible representations of operator algebras. Bulletin of the American Mathematical Society, 53 (2), 73-88.\\
$[16]$ Doran, R. S.; Belfi, V. A. (1986). Characterizations of C$^*$-algebras: the Gelfand-Naimark Theorems. CRC Press\\
$[17]$ Atiyah, M. F. (1989). K-theory. Advanced Book Classics (2nd ed.); Addison-Wesley\\
$[18]$ Bredon, G. E. (1997). Sheaf theory. Graduate Texts in Mathematics, 170 (2nd ed.), Berlin, New York: Springer-Verlag\\
$[19]$ Strickland, N. (2020). Algebraic theory of abelian groups. Preprint (arXiv:2001.10469)\\
$[20]$ Ara\'ujo, J.; Cameron, P. J.; Mitchell, J.; Neunhoffer, M. (2012). The classification of normalizing groups\\
$[21]$ Tahar, G. (2017). Ordered algebraic structures and classification of semifields. Preprint (	arXiv:1709.06923)\\
$[22]$ R\'ua, I. F.; Combarro, E. F.; Ranilla, J. (2009). Classification of semifields of order 64. J. of Algebra. Vol 322, Issue 11, 4011-4029\\
$[23]$ Winkler, P. M. (1980). Classification of algebraic structures by work space. Algebra Universalis 11, 320-333\\
$[24]$ Chapoton, F. (2004). Classification of some simple graded pre-Lie algebras of growth one. Commun. Algebra. 32(1): 243-251\\
$[25]$ Xu, C. (2019). Compatible left-symmetric algebraic structures on high rank Witt and Virasoro algebras. Preprint (arXiv:1910.13644)\\
$[26]$ Wu, H.; Yuan, L. (2017). Classification of finite irreducible conformal modules over some Lie conformal algebras related to the Virasoro conformal algebra. J. Math. Phys. 58(4): 041701\\
$[27]$ Kong, X.; Chen, H.; Bai, C. (2011). Classification of graded left-symmetric algebraic structures on Witt and Virasoro algebras. Int. J. Math. 22(2): 201-222\\
$[28]$ Ismailov, N.; Kaygorodov, I.; Mashurov, F. (2020). The Algebraic and Geometric Classification of Nilpotent Assosymmetric Algebras. Algebr Represent Theor\\
$[29]$ Abdelwahab, H.; Calderón, A.J.; Kaygorodov, I. (2019). The algebraic and geometric classification of nilpotent binary Lie algebras. International Journal of Algebra and Computation 29(6), 1113-1129\\
$[30]$ Calder\'on Mart\'in, A.; Fern\'andez Ouaridi, A.; Kaygorodov; I. (2019). The classification of $n$-dimensional anticommutative algebras with $(n-3)$-dimensional annihilator. Communications in Algebra 47(1), 173-181\\
$[31]$ Calder\'on Mart\'in, A.; Fern\'andez Ouaridi, A.; Kaygorodov, I. (2018). The classification of $2$-dimensional rigid algebras, Linear and Multilinear Algebra\\
$[32]$ Darijani, I.; Usefi, H. (2016). The classification of $5$-dimensional $p$-nilpotent restricted Lie algebras over perfect fields, I. Journal of Algebra 464, 97-140\\
$[33]$ De Graaf, W. (2007). Classification of $6$-dimensional nilpotent Lie algebras over fields of characteristic not $2$. Journal of Algebra 309(2), 640-653\\
$[34]$ De Graaf, W. (2018). Classification of nilpotent associative algebras of small dimension. International Journal of Algebra and Computation 28(1), 133-161\\
$[35]$ Demir, I.; Misra, K.; Stitzinger, E. (2017). On classification of four-dimensional nilpotent Leibniz algebras. Communications in Algebra 45(3), 1012-1018\\
$[36]$ Camacho, L. M.; Karimjanov, I.; Kaygorodov, I.; Khudoyberdiyev, A. (2020). One-generated nilpotent Novikov algebras, Linear and Multilinear Algebra\\
$[37]$ Bengtsson, I.; Zyczkowski, K. (2006). Complex projective spaces. In Geometry of Quantum States: An Introduction to Quantum Entanglement (pp. 102-134). Cambridge: Cambridge University Press\\
$[38]$ Besse, A. L. (1978). Manifolds all of whose geodesics are closed, Ergebnisse der Mathematik und ihrer Grenzgebiete (Results in Mathematics and Related Areas), 93, Berlin, New York: Springer-Verlag\\
$[39]$ S. Akbulut. (1994). On quotients of complex surfaces under complex conjugation. J. Reine Angew.
Math. 447, 83-90\\
$[40]$ V. I. Arnold. (1998). The branched covering $\mathbb{C}P \rightarrow S^4$, hyperbolocity and projective topology, Sibirsk. Mat. Zh. 29, no. 5, 36-47\\
$[41]$ V. Kharlamov. (1976). The topological type of nonsingular surfaces in $\mathbb{RP}^3$ of degree $4$, Funkc. Anal. i Prilozhen. 10, no. 4, 55-68 (Russian); English transl. in Functional Anal. Appl. 10, no. 4, 295-305\\
$[42]$  N. Kuiper. (1974). The quotient space of $\mathbb{CP}^2$ by complex conjugation is the $4$-sphere, Math. Ann. 208, 175-177\\
$[43]$  V. Kharlamov (1972). The maximal number of components of a $4$th degree surface in $\mathbb{RP}^3$, Funkc. Anal. i Prilozhen. 6, no. 4, 101\\
$[44]$ V. Kharlamov. (1984). On the classification of non-singular surfaces of degree $4$ in $\mathbb{RP}^3$ with respect to rigid isotopies, Funct. Anal. Appl. 18, no. 1, 49-56\\
$[45]$ V. Kharlamov. (1981). On a number of components of an M-surface of degree $4$ in $\mathbb{RP}^3$, Proc. of XVI soviet algebraic conference, Leningrad, 353-354\\
$[46]$ Degtyarev, A.; Kharlamov, V. (2000). Topological properties of the real algebraic varieties: du cote chez rokhlin\\
$[47]$ Botvinnik, B. (2010). Notes on the course 'Algebraic Topology'. University of Oregon\\
$[48]$ Chan, V. (2013). Topological K-Theory of Complex Projective Spaces. Preprint (arXiv:1303.3959)\\
$[49]$ Conrad, B. (1973). On Manifolds with the Homotopy Type of Complex Projective Space. Transactions of the American Mathematical Society, 176, 165-180\\
$[50]$ Brasselet, J. P. (2007). Singularities in geometry and topology: proceedings of the Trieste Singularity Summer School and Workshop, ICTP, Trieste, Italy, 15 August - 3 September 2005. World Scientific Publishing Company Pte Limited\\
$[51]$ Carlstr\"om, Jesper (2004). Wheels - On Division by Zero. Mathematical Structures in Computer Science, Cambridge University Press, 14 (1): 143-184\\
$[52]$ Lundell, A. T.; Weingram, S. (1970). The topology of CW-complexes. Van Nostrand University Series in Higher Mathematics\\
$[53]$ Hatcher, Allen (2002). Algebraic topology. Cambridge University Press\\
$[54]$ Elter, E.; Lienhardt. P. (1994). Cellular complexes as structured semi-simplicial sets. International Journal of Computational Geometry $\&$ Applications, 1(2):191-217\\
$[55]$ Wall, C. (1965). Finiteness Conditions for CW-Complexes. Annals of Mathematics, 81(1), second series, 56-69\\
$[56]$ Wall, C. T. (1966). Finiteness conditions for CW complexes. II. Proceedings of the Royal Society of London. Series A. Mathematical and Physical Sciences, 295, 129-139\\
$[57]$ Whitehead, J. H. C. (1949). Combinatorial homotopy. I., Bull. Amer. Math. Soc., 55, 213-245\\
$[58]$ Whitehead, J. H. C. (1949). Combinatorial homotopy. II., Bull. Amer. Math. Soc., 55, 453-496\\
$[59]$ Milnor, J. (1959). On spaces having the homotopy type of a CW-complex. Trans. Amer. Math. Soc. 90: 272-280\\
$[60]$ L\"uck, W. (2005). Survey on classifying spaces for families of subgroups. In Infinite groups: geometric, combinatorial and dynamical aspects, volume 248 of Progr. Math., pages 269-322. Birkh\"auser, Basel\\
$[61]$ Moerdijk, I. (1995). Classifying spaces and classifying topoi, Lecture Notes in Math. 1616, Springer-Verlag, New York\\
$[62]$ Milnor, J. W.; Stasheff, J. D. (1974). Characteristic classes, Princeton Univ. Press\\
$[63]$ Madsen, J.; Milgram, R. J. (1979). The classifying spaces for surgery and cobordism of manifolds, Princeton Univ. Press\\
$[64]$ Baker, A.; Clarke, F.; Ray, N.; Schwartz, L. (1989). On the Kummer congruences and the stable homotopy of $BU$, Trans. Amer. Math. Soc., American Mathematical Society, 316 (2): 385-432\\
$[65]$ Adams, J. F. (1974). Stable Homotopy and Generalised Homology, University Of Chicago Press\\
$[66]$ Eilenberg, S.; MacLane, S. (1945). Relations between homology and homotopy groups of spaces. Ann. of Math., 46, 480-509\\
$[67]$ Eilenberg, S.; MacLane, S. (1950). Relations between homology and homotopy groups of spaces. II. Ann. of Math., 51, 514-533\\
$[68]$ Adhikari, M. R. (2016). Eilenberg-MacLane Spaces. In: Basic Algebraic Topology and its Applications. Springer, New Delhi\\
$[69]$ Anderson, D. W.; Hodgkin, L. (1968). The K-theory of Eilenberg-Maclane complexes, Topology, Volume 7, Issue 3, 317-329\\
$[70]$ Kleiman, S. (1969). Geometry on grassmanians and applications to splitting bundles and smoothing
cycles, Publ. Math. I.H.E.S. 36, 281-297\\
$[71]$ Zanella, C. (1995). Embeddings of Grassmann spaces. J Geom 52, 193-201\\
$[72]$ Bichara, A.; Tallini, G. (1982). On a characterization of the Grassmann manifold representing the planes in a projective space. Combinatorial and geometric structures and their applications (Trento, 1980), Ann. Discr. Math.14, 129-149\\
$[73]$ Bichara, A.; Tallini, G. (1983). On a characterization of Grassmann space representing the h-dimensional subspaces in a projective space. Combinatorics '81 (Rome, 1981), Ann. Discr. Math.18, 113-131\\
$[74]$ Tallini, G. (1981). On a characterization of the Grassmann manifold representing the lines in a projective space. Finite geometries and designs (Chelwood Gate, 1980), London Math. Soc. Lecture Note Ser., 49, Cambridge University Press, 354-358\\
$[75]$ Wells, A. L. Jr. (1983). Universal projective embeddings of the Grassmannian, half spinor, and dual orthogonal geometries. Quart. J. Math. Oxford (2)34, 375-386\\
$[76]$ Licata, D. R.; Finster, E. (2014). Eilenberg-MacLane spaces in homotopy type theory. CSL-LICS '14

\end{document}